\documentclass[12pt]{article}
\usepackage{amssymb,amsmath,amsthm,amsfonts,mathrsfs}
\usepackage{enumerate}
\newtheorem{Th}{Theorem}[section]

\theoremstyle{definition}
\newtheorem{Def}[Th]{Definition}

\usepackage{pgf,tikz}
\usetikzlibrary{intersections}
\usepackage{graphicx}
\usepackage{float}

\usepackage[colorlinks=true, allcolors=blue]{hyperref}
\usepackage[all]{hypcap}
\usepackage[capitalize,nameinlink]{cleveref}

\usepackage{titlesec}
\titlelabel{\thetitle.\quad}

\newtheorem{theo}{Theorem}

\theoremstyle{definition}

\begin{document}
\title{ Bipartite Powers of Certain Classes of Bipartite Graphs}
\def\correspondingauthor{\footnote{Corresponding author}}
\author{  Indrajit Paul, Ashok Kumar Das\correspondingauthor{}\\Department of Pure Mathematics, University of Calcutta\\
Email Address -  
paulindrajit199822@gmail.com \&\\ ashokdas.cu@gmail.com}

\maketitle

\begin{abstract}
The concept of graph powers has been extensively studied in graph theory. Analogous to graph powers, Chandran et al. \cite{2} introduced the notion of bipartite powers for bipartite graphs. In this paper, we show that the class of interval bigraphs, as well as the class of proper interval bigraphs are closed under the operation of taking bipartite powers. Finally, we define strongly closed property for bipartite graphs under powers and have shown that the class of chordal bipartite graphs is strongly closed under bipartite powers.

\end{abstract}
\section{Introduction}
\indent\hspace{0.3in} Consider a simple, finite, and connected graph $G = (V, E)$. For any pair of vertices $u$ and $v$ in $G$, let $d_G(u, v)$ denote the distance between $u$ and $v$ in $G$, which is the length of the shortest path from $u$ to $v$ in $G$. For a positive integer $k$, the $k$th power of $G$, denoted as $G^k$, is a graph with the same vertex set as $G$. In $G^k$, two vertices $u$ and $v$ are adjacent if and only if $d_G(u, v)\leq k$. A graph class $\mathcal{C}$ is said to be \textit{closed under powers} if $G \in\mathcal{C}$  implies that $G^k \in \mathcal{C}$ for any positive integer $k$. Also a graph class $\mathcal{C}$ is said to be \textit{strongly closed under powers} if for $k\in \mathbf{N}$, if $G^k\in \mathcal{C}$ then $G^{k+1}\in \mathcal{C}$.
\par Powers of graphs have been studied in different aspects. Motwani and Sudan\cite{mot-su}, and Brandst$\ddot{a}$dt et al.\cite{bran-drag-xian-yan}, Bandelt et al.\cite{band-hen-nic} respectively studied power graphs from algorithmic and other aspects. But it is interesting to the researchers to investigate which classes of graphs are closed and strongly closed under powers. Raychaudhuri \cite{Raychaudhuri-1, Raychaudhuri-2}, Flotow \cite{flotow} and others have demonstrated that several classes of graphs, like interval graphs, circular arc graphs, chordal graphs, comparability graphs are closed under powers. Raychaudhuri and Flotow also proved that interval graphs, proper interval graphs, trapezoid graphs are strongly closed under powers. In a recent paper Das and Paul \cite{16} demonstrated that circular arc graphs and proper circular arc graphs are strongly closed under powers.\\

Now, let’s turn our attention to bipartite graphs, often referred to as bigraphs, denoted as $B = (X, Y, E)$, where the vertex set of $B$ is partitioned into two stable sets, $X$ and $Y$. In a bigraph, one endpoint of an edge belongs to the $X$ partite set, while the other belongs to the $Y$ partite set.
 
 \par It is natural to extend the notion of graph powers for the class of bipartite graphs.\\
 It is well known that bipartite graphs do not contain any odd cycle. But the even powers of a bigraph contains odd cycles hence they are not bipartite graphs. Thus, when we define $B^k$ for a bigraph $B$, $k$ must be odd.\\

Considering this fact in mind,  Chandran et al.\cite{2} introduced the concept of bipartite power for the bigraphs. For a bipartite graph $B$, the bipartite power graph $B^{[k]}$ is defined as follows: it shares the same vertex set as $B$, and two vertices in $B^{[k]}$ are adjacent if their distance in $B$ is an odd natural number and less than or equal to $k$. They demonstrated that for any tree $T$, $T^{[k]}$ is a chordal bipartite graph when $k$ is odd. Chandran and Mathew \cite{1} further strengthened this result, proving that the class of chordal bipartite graphs is indeed closed under powers.\\

 Okamoto et al.\cite{10} established that interval bigraphs and bipartite permutation graphs also exhibit closure properties under bipartite powers. It is known that the class of bipartite permutation graphs is same as the class of proper interval bigraphs \cite{g.s}. Thus the result of Okamoto et al.\cite{10} shows the closure property of proper interval graphs under power. However, the way they defined intervals for the bigraph $B^{[k]}$ is not well defined and requires some modification. In this paper, we modify the construction of intervals for the bigraph $B^{[k]}$ and prove that interval bigraphs are closed under bipartite powers.

 \par A $(0,1)$-matrix is said to have monotone consecutive arrangement (MCA) if there exist independent row and column permutations  exhibiting the structure that 0’s in the resulting matrix can be labeled R or C such that every position above and to the right of an R is also labeled R, and every position below and to the left of a C is also labeled C.\\
 A bigraph is a proper interval bigraph if and only if its biadjacency matrix has a MCA. Using this, in the present paper  we shall give a new and elegant proof that proper interval bigraphs are closed under bipartite powers. \\
 Analogous to the strongly closed property of a graph class under powers, we define a class of bigraphs $\mathcal{C}$ is \textit{strongly closed} under powers if for any odd positive integer $k$, if $G^{[k]}\in \mathcal{C}$ then $G^{[k+2]}\in \mathcal{C}$. \\
 Finally, we give a simple proof that chordal bipartite graphs are strongly closed under bipartite powers, which is the main result of this paper.

 \section{Preliminaries}

\indent\hspace{0.3in}
  A bigraph $B = (X, Y, E)$ is said to be an interval bigraph if there exists a one-to-one correspondence between the vertex set $X \cup Y$ of $B$ and a collection of intervals $\{I_v : v \in X \cup Y \}$ such that $xy \in E$ if and only if $I_x \cap I_y \neq \emptyset$.

The biadjacency matrix of a bipartite graph is a submatrix of its adjacency matrix, which consists of rows indexed by the vertices of one partite set and columns indexed by the other.\\
A proper interval bigraph is an interval bigraph if in the bipartite interval representation no interval is properly contained in another. A unit interval bigraph is an interval bigraph if all intervals in the representations are of unit length. Sen an Sanyal \cite{sensanyal} introduced these two subclasses of interval bigraphs and also showed that these two subclasses of interval bigraphs are equivalent via their biadjacency matrices. In this connection they state the following definition.
\begin{Def}
    A $0$,$1$- matrix with $n$ non-zero rows and $m$ non-zero columns has a monotone consecutive arrangement (MCA) if and only if it has independent row and column permutations such that $1$'s appear consecutively in each row and the values $\{a_i\}$ and $\{b_i\}$ denoting the initial column and final column of the intervals of $1$'s in row $i$ satisfy $a_1\leq a_2\leq\dots \leq a_n$ and $b_1\leq b_2\leq\dots \leq b_n$.
\end{Def}
\par An equivalent definition of MCA of a $n\times m$ $(0,1)$-matrix is the following: if and only if the matrix has independent row and column permutations such that the $1$'s appear consecutively in each column and the values $\{c_j\}$ and $\{d_j\}$ denoting the initial row and final row of the intervals of $1$'s in row $j$ satisfy $c_1\leq c_2\leq\dots \leq c_m$ and $d_1\leq d_2\leq\dots \leq d_m$.

An alternative definition of monotone consecutive arrangement (MCA) of a $(0,1)$- matrix, i.e., binary matrix, is the following. A binary matrix is said to have a monotone consecutive arrangement (MCA) if there exist independent row and column permutations of the matrix exhibiting the following structure (Figure 1). The $0$s of the resulting matrix can be labeled R or C such that every entry above and right of an R is an R, and every entry below and left of C is also a C.
\begin{figure}[H]
     \centering
     \begin{tikzpicture}[line cap=round,line join=round,x=1.0cm,y=1.0cm,scale=1]
\clip(2.,-1.) rectangle (11.,7.);
\draw [line width=1.pt] (3.,6.)-- (10.,6.);
\draw [line width=1.pt] (3.,6.)-- (3.,0.);
\draw [line width=1.pt] (10.,6.)-- (10.,0.);
\draw [line width=1.pt] (3.,0.)-- (10.,0.);
\draw [line width=1. pt] (3.,6.)-- (5.,6.);
\draw [line width=1.pt] (5.,6.)-- (5.,5.);
\draw [line width=1.pt] (5.,5.)-- (6.,5.);
\draw [line width=1.pt] (6.,5.)-- (6.,4.);
\draw [line width=1.pt] (6.,4.)-- (8.,4.);
\draw [line width=1.pt] (8.,4.)-- (8.,2.);
\draw [line width=1.pt] (8.,2.)-- (9.,2.);
\draw [line width=1.pt] (9.,2.)-- (9.,1.);
\draw [line width=1.pt] (9.,1.)-- (10.,1.);
\draw [line width=1.pt] (10.,1.)-- (10.,0.);
\draw [line width=1.pt] (3.,6.)-- (3.,4.);
\draw [line width=1.pt] (3.,4.)-- (4.,4.);
\draw [line width=1.pt] (4.,4.)-- (4.,3.);
\draw [line width=1.pt] (4.,3.)-- (5.,3.);
\draw [line width=1.pt] (5.,3.)-- (5.,2.);
\draw [line width=1.pt] (5.,2.)-- (6.,2.);
\draw [line width=1.pt] (6.,2.)-- (6.,1.);
\draw [line width=1.pt] (6.,1.)-- (7.,1.);
\draw [line width=1.pt] (7.,1.)-- (7.,0.);
\draw [line width=1.pt] (7.,0.)-- (10.,0.);
\draw (3.2,5.7) node[anchor=north west] {$1$};
\draw (4.2,5.7) node[anchor=north west] {$1$};
\draw (3.2,4.7) node[anchor=north west] {$1$};
\draw (4.2,4.7) node[anchor=north west] {$1$};
\draw (5.2,4.7) node[anchor=north west] {$1$};
\draw (4.2,3.7) node[anchor=north west] {$1$};
\draw (5.2,3.7) node[anchor=north west] {$1$};
\draw (6.2,3.7) node[anchor=north west] {$1$};
\draw (7.2,3.7) node[anchor=north west] {$1$};
\draw (5.2,2.7) node[anchor=north west] {$1$};
\draw (6.2,2.7) node[anchor=north west] {$1$};
\draw (7.2,2.7) node[anchor=north west] {$1$};
\draw (6.2,1.7) node[anchor=north west] {$1$};
\draw (7.2,1.7) node[anchor=north west] {$1$};
\draw (8.2,1.7) node[anchor=north west] {$1$};
\draw (7.2,0.7) node[anchor=north west] {$1$};
\draw (8.2,0.7) node[anchor=north west] {$1$};
\draw (9.2,0.7) node[anchor=north west] {$1$};
\draw (8.2,4.7) node[anchor=north west] {$0_R$};
\draw (5.2,5.7) node[anchor=north west] {$0_R$};
\draw (6.2,5.7) node[anchor=north west] {$0_R$};
\draw (7.2,5.7) node[anchor=north west] {$0_R$};
\draw (8.2,5.7) node[anchor=north west] {$0_R$};
\draw (9.2,5.7) node[anchor=north west] {$0_R$};
\draw (6.2,4.7) node[anchor=north west] {$0_R$};
\draw (7.2,4.7) node[anchor=north west] {$0_R$};
\draw (9.2,4.7) node[anchor=north west] {$0_R$};
\draw (8.2,3.7) node[anchor=north west] {$0_R$};
\draw (9.2,3.7) node[anchor=north west] {$0_R$};
\draw (8.2,2.7) node[anchor=north west] {$0_R$};
\draw (9.2,2.7) node[anchor=north west] {$0_R$};
\draw (9.2,1.7) node[anchor=north west] {$0_R$};
\draw (3.2,3.7) node[anchor=north west] {$0_C$};
\draw (3.2,2.7) node[anchor=north west] {$0_C$};
\draw (4.2,2.7) node[anchor=north west] {$0_C$};
\draw (3.2,1.7) node[anchor=north west] {$0_C$};
\draw (4.2,1.7) node[anchor=north west] {$0_C$};
\draw (5.2,1.7) node[anchor=north west] {$0_C$};
\draw (3.2,0.7) node[anchor=north west] {$0_C$};
\draw (4.2,0.7) node[anchor=north west] {$0_C$};
\draw (5.2,0.7) node[anchor=north west] {$0_C$};
\draw (6.2,0.7) node[anchor=north west] {$0_C$};
\draw (2.2,5.7) node[anchor=north west] {$x_1$};
\draw (2.2,4.7) node[anchor=north west] {$x_2$};
\draw (2.2,3.7) node[anchor=north west] {$x_3$};
\draw (2.2,2.7) node[anchor=north west] {$x_4$};
\draw (2.2,1.7) node[anchor=north west] {$x_5$};
\draw (2.2,0.7) node[anchor=north west] {$x_6$};
\draw (3.2,6.7) node[anchor=north west] {$y_1$};
\draw (4.2,6.7) node[anchor=north west] {$y_2$};
\draw (5.2,6.7) node[anchor=north west] {$y_3$};
\draw (6.2,6.7) node[anchor=north west] {$y_4$};
\draw (7.2,6.7) node[anchor=north west] {$y_5$};
\draw (8.2,6.7) node[anchor=north west] {$y_6$};
\draw (9.2,6.7) node[anchor=north west] {$y_7$};
\end{tikzpicture}
     \caption{Monotone consecutive arrangement of a $(0,1)$ matrix.}
     \label{fig_1}
 \end{figure}

In the following theorem Sen and Sanyal \cite{sensanyal} characterized proper interval bigraphs in terms of monotone consecutive arrangement of their biadjacency matrices, also showed that the class of proper interval bigraphs and unit interval bigraphs are equivalent.
\begin{theo}[\cite{sensanyal}]
    For a bipartite graph $B$ the following statements are equivalent:
    \begin{itemize}
    \item[(i).] $B$ is an unit interval bigraph.
        \item[(ii).] $B$ is a proper interval bigraph.
        \item[(iii).] Biadjacency matrix $A(B)$ of $B$ has a MCA.
    \end{itemize}
\end{theo}

\section{Main Result}
\subsection{Bipartite powers of interval bigraphs}

A bigraph $B=(X,Y,E)$ is an interval bigraph if there exists a family $\mathcal{I}=\{I_v: v\in X\cup Y\}$ of intervals such that $xy\in E$ if and only if $I_x\cap I_y\neq \emptyset$. If $B=(X,Y,E)$ is an interval bigraph, then we have an interval representation $B$ such that left end points of the intervals $I_x$, $x\in X$ are in increasing order. Also the left end points of the intervals $I_y$, $y\in Y$ are in increasing order. Bipartite powers of interval bigraphs have defined by Okamoto et al. \cite{10} in the following way. Let $k$ be an odd natural number. Then for $B^{[k]}$, they form new interval $I_k(x)=[l(x), r_k(x)]$, where $r_k(x)=$ max $\{l(y): d_G(x,y)\leq k\}$ and $I_k(y)=[l(y), r_k(y)]$, where $r_k(y)=$ max $\{l(x): d_G(x,y)\leq k\}$.\\
As previously mentioned, this definition for the intervals of $B^{[k]}$ is not well defined and required some modification. And this becomes evident when examining the following interval bigraphs. Let's consider an interval bigraph with the interval representation. 
\begin{figure}[H]
    \centering
    \begin{tikzpicture}[scale=.45, line cap=round,line join=round,,x=1.0cm,y=1.0cm]
\clip(3.,-15.) rectangle (30.,15.);
\draw [line width=.2pt] (12.,12.)-- (18.,12.);
\draw [line width=.2pt] (12.,12.)-- (10.,8.);
\draw [line width=.2pt] (18.,12.)-- (20.,8.);
\draw [line width=.2pt] (10.,8.)-- (12.,4.);
\draw [line width=.2pt] (12.,4.)-- (18.,4.);
\draw [line width=.2pt] (20.,8.)-- (18.,4.);
\draw [line width=.2pt] (10.,8.)-- (20.,8.);
\draw [line width=.2pt] (12.,4.)-- (18.,12.);
\draw [line width=.2pt] (12.,12.)-- (18.,4.);
\draw [line width=.2pt] (18.,4.)-- (20.,2.);
\draw [line width=.2pt] (12.,4.)-- (10.,2.);
\draw [line width=.2pt] (20.,8.)-- (22.,8.);
\draw [line width=.2pt] (10.,8.)-- (8.,8.);
\draw [line width=.2pt] (18.,12.)-- (20.,14.);
\draw [line width=.2pt] (12.,-4.)-- (20.,-4.);
\draw [line width=.2pt] (8.,-6.)-- (20.,-6.);
\draw [line width=.2pt] (14.,-8.)-- (22.,-8.);
\draw [line width=.2pt] (10.,-8.)-- (12.,-8.);
\draw [line width=.2pt] (16.,-10.)-- (18.,-10.);
\draw [line width=.2pt] (18.,-12.)-- (20.,-12.);
\draw [line width=.2pt] (8.,-10.)-- (14.,-10.);
\draw [line width=.2pt] (14.,-12.)-- (16.,-12.);
\draw [line width=.2pt] (20.,-14.)-- (22.,-14.);
\draw [line width=.2pt] (22.,-12.)-- (24.,-12.);
\draw [line width=.2pt] (4.,-14.)-- (8.,-14.);
\draw (15.5,-2.7) node[anchor=north west] {$x_1$};
\draw (13.5,-4.8) node[anchor=north west] {$x_2$};
\draw (17,-6.7) node[anchor=north west] {$x_3$};
\draw (10.5,-6.7) node[anchor=north west] {$x_4$};
\draw (16.5,-8.7) node[anchor=north west] {$x_5$};
\draw (18.5,-10.7) node[anchor=north west] {$x_6$};
\draw (10.5,-8.7) node[anchor=north west] {$y_1$};
\draw (14.5,-10.7) node[anchor=north west] {$y_2$};
\draw (20.5,-12.7) node[anchor=north west] {$y_3$};
\draw (5.5,-12.7) node[anchor=north west] {$y_4$};
\draw (22.5,-10.7) node[anchor=north west] {$y_5$};
\draw (3.5,-0.2) node[anchor=north west] {0};
\draw (5.5,-0.2) node[anchor=north west] {1};
\draw (7.5,-0.2) node[anchor=north west] {2};
\draw (9.5,-0.2) node[anchor=north west] {3};
\draw (11.5,-0.2) node[anchor=north west] {4};
\draw (13.5,-0.2) node[anchor=north west] {5};
\draw (15.5,-0.2) node[anchor=north west] {6};
\draw (17.5,-0.2) node[anchor=north west] {7};
\draw (19.5,-0.2) node[anchor=north west] {8};
\draw (21.5,-0.2) node[anchor=north west] {9};
\draw (23.5,-0.1) node[anchor=north west] {10};
\draw (11.5,13.6) node[anchor=north west] {$x_1$};
\draw (17.,13.6) node[anchor=north west] {$y_1$};
\draw (20.,14.35) node[anchor=north west] {$x_4$};
\draw (7.5,9.5) node[anchor=north west] {$x_6$};
\draw (9.,8.) node[anchor=north west] {$y_3$};
\draw (10,4.5) node[anchor=north west] {$x_3$};
\draw (18.2,4.5) node[anchor=north west] {$y_2$};
\draw (20.,8.) node[anchor=north west] {$x_2$};
\draw (22.550438118434776,8.681256018158766) node[anchor=north west] {$y_4$};
\draw (8.5,2.5) node[anchor=north west] {$y_5$};
\draw (20.2,2.5) node[anchor=north west] {$x_5$};
\begin{scriptsize}
\draw [fill=black] (12.,12.) circle (2.5pt);
\draw [fill=black] (18.,12.) circle (2.5pt);
\draw [fill=black] (10.,8.) circle (2.5pt);
\draw [fill=black] (20.,8.) circle (2.5pt);
\draw [fill=black] (12.,4.) circle (2.5pt);
\draw [fill=black] (18.,4.) circle (2.5pt);
\draw [fill=black] (20.,2.) circle (2.5pt);
\draw [fill=black] (10.,2.) circle (2.5pt);
\draw [fill=black] (22.,8.) circle (2.5pt);
\draw [fill=black] (8.,8.) circle (2.5pt);
\draw [fill=black] (20.,14.) circle (2.5pt);
\draw [fill=black] (4.,-2.) circle (2.pt);
\draw [fill=black] (6.,-2.) circle (2.pt);
\draw [fill=black] (8.,-2.) circle (2.pt);
\draw [fill=black] (10.,-2.) circle (2.pt);
\draw [fill=black] (12.,-2.) circle (2.pt);
\draw [fill=black] (14.,-2.) circle (2.pt);
\draw [fill=black] (16.,-2.) circle (2.pt);
\draw [fill=black] (18.,-2.) circle (2.pt);
\draw [fill=black] (20.,-2.) circle (2.pt);
\draw [fill=black] (22.,-2.) circle (2.pt);
\draw [fill=black] (24.,-2.) circle (2.pt);

\end{scriptsize}
\end{tikzpicture}
    \caption{An interval bigraph $B$ with its interval representation.}
    \label{fig2}
\end{figure}
In this example, according to the definition of Okamoto et al. in $B^{[3]}$, $I_3(y)=[8,7]$, which is not possible. Now we shall prove that if $B$ is an interval graph then $B^{[k]}$ is also an interval graph for any odd natural number $k$.
\begin{theo}\label{t3}
    For an odd natural number $k$, $I_k(x)\cap I_k(y)\neq \emptyset$ if and only if $d(x,y)\leq k$ in $B$.
\end{theo}

\begin{proof}
    Let $d(x,y)\leq k$. Already we have mentioned that we can take the  interval representation of $B$ such that left end points of $I(x_i)$, $ x_i\in X$ are in the increasing order as the increasing order of their suffixes. Similar relation holds for the left end points of $I(y_j)$, $y_j\in Y $. Then from the definitions of $r_k(x)$ and $r_k(y)$, we have $l(x)\leq r_k(y)$ and $l(y)\leq r_k(x)$. This implies $I_k(x)\cap I_k(y)\neq \emptyset$. Also, if  $r_k(x)=l(x)$ and $r_k(y)=l(y)$, then $l(x)\leq l(y)$ and $l(y)\leq l(x)$. In the other case, suppose $l(x)\leq r_k(y)=l(y)$ and $l(y)\leq r_k(x)\neq l(x)$. Thus in these cases also we have $I_k(x)\cap I_k(y)\neq\emptyset$.\\
    \par For the converse, suppose $I_k(x)\cap I_k(y)\neq\emptyset$. Thus $l(x)\leq r_k(y)$ and $l(y)\leq r_k(x)$. Now we consider the following possibilities.
    \\ \textit{\textbf{Case 1:}} Let $r_k(y)=l(y)$ and $r_k(x)=l(x)$. Then $l(x)\leq l(y)$ and $l(y)\leq l(x)$, i.e. $l(x)=l(y)$. So  $I(x)\cap I(y)\neq\emptyset$ i.e. $xy\in E$. Thus $d(x,y)\leq k$.\\
    \textit{\textbf{Case 2:}} Let $r_k(y)=$ max $\{ l(x_i): d(x_i, y)\leq k \}=l(x_r)$. and $r_k(x)=$ max $\{ l(y_j): d(x, y_j)\leq k\}=l(y_s)$ i.e. $l(x)\leq l(x_r)$ and $l(y)\leq l(y_s)$. Without loss of generality, suppose $l(x)\leq l(y)$. Now from the definition of $l(y_s)$, for any $q > s, l(y_q)> l(y_s)$. Again as  $I_k(x)\cap I_k(y)\neq\emptyset$, we have $l(x)\leq l(y)< l(y_s)$. Hence $d(x,y)\leq k$.
\end{proof}
\subsection{Bipartite powers of proper interval bigraphs}
In this section, using MCA (defined before) of  the biadjacency matrix of proper interval bigraphs, we shall prove that the class of proper interval bigraphs are closed under bipartite powers.
\begin{theo} Let $B=(X,Y,E)$ be a proper interval bigraph. Then $B^{[k]}$ is also a proper interval bigraph, where $k$ is an odd positive integer.
    
\end{theo}

\begin{proof}
Since $B$ is proper interval bigraph, its biadjacency matrix $A(B)$ has monotone consecutive arrangement. We shall show that the biadjacency matrix $A(B^{[k]})$ of $B^{[k]}$ has also monotone consecutive arrangement (MCA). Let $A(B)$ be a $m\times n$ $0$, $1$- matrix and rows and columns of $A(B)$ are arranged such that it has an MCA structure. Let $a_i$, $b_i$ $(1\leq i\leq n)$ denote respectively the initial and final column of the intervals of $1$'s of the $i$-th row. Also $c_j$ and $d_j$ $(i\leq j\leq m)$ denote respectively the initial and final row of the intervals of $1$'s of the $j$-th column.\\
Now we consider four mapping   $\alpha$ : $\{1,2,...,n\}$ $\rightarrow{\{1,2,...,m\}}$, $\beta$ : $\{1,2,...,n\}$ $\rightarrow{\{1,2,...,m\}}$, $\gamma$ : $\{1,2,...,m\}$ $\rightarrow{\{1,2,...,n\}}$, $\delta$ : $\{1,2,...,m\}$ $\rightarrow{\{1,2,...,n\}}$ respectively defined by $\alpha (i)=a_i$, $\beta (i)=b_i$ $(1\leq i\leq n)$, $\gamma (j)=c_j$ and $\delta (j)=d_j$ $(1\leq j\leq m)$. Clearly all the composite mapping $\alpha\gamma$, $\gamma\alpha$, $\beta\delta$ and $\delta\beta$ are well-defined.\\
Now we shall prove that the biadjacency matrix $A(B^{[k]})$ of $B^{[k]}$ is also exhibits MCA, where $k$ is an odd natural number.\\
Let in $B^{[k]}$ the vertices $x_i$ and $y_j$ are adjacent. Thus there exists a shortest $x_i$-$y_j$ path of odd length $\leq k$ in 
$B$. Let $A(B)$ has MCA and $(x_i,y_j)$ entry below the lower stair, we shall show that all the entries to the right and above the $(x_i,y_j)$ entry are also $1$ up to the lower stair of MCA. Now, the $x_i$-$y_j$ shortest path in $B$ can be written as\\
$P$: $x_i$ $y_{\alpha(i)}$ $x_{\gamma\alpha(i)}$ $y_{\alpha\gamma\alpha(i)}$ $...$ $x_\frac{\alpha\gamma\alpha\gamma...\alpha(i)}{ k_1 times}$ $y_{j}$, where $k_1+1$ is an odd integer and $k_1+1\leq k$. Next, the $x_i$-$y_{j+1}$ shortest path is\\
$P_1$: $x_i$ $y_{\alpha(i)}$ $x_{\gamma\alpha(i)}$ $y_{\alpha\gamma\alpha(i)}$ $...$ $x_\frac{\alpha\gamma\alpha\gamma...\alpha(i)}{ k_2 times}$ $y_{j+1}$, where $k_2+1$ is an odd integer and $k_2+1\leq k_1+1$. Hence $P_1$ is a subpath of $P$.\\
 Next, the $x_{i-1}$-$y_{j}$ shortest path is\\
$P_2$: $x_{i-1}$ $y_{\alpha(i-1)}$ $x_{\gamma\alpha(i-1)}$ $y_{\alpha\gamma\alpha(i-1)}$ $...$ $x_\frac{\alpha\gamma\alpha\gamma...\alpha(i-1)}{ k_3 times}$ $y_{j}$, where $k_3+1$ is an odd integer and $k_3+1\leq k$.\\
Hence in $A(B^{[k]})$ the biadjacency matrix of $B^{[k]}$, if the entry $(x_i,y_j)$ is a $1$ then the entries $(x_i,y_{j+1})$ and $(x_{i-1},y_j)$ are also $1$.
\begin{figure}[H]
    \centering
    \begin{tikzpicture}[line cap=round,line join=round,x=1.0cm,y=1.0cm]
\clip(3.,0.) rectangle (17.,11.);
\draw [line width=1.pt] (4.,10.)-- (4.,1.);
\draw [line width=1.pt] (4.,10.)-- (16.,10.);
\draw [line width=1.pt] (4.,1.)-- (16.,1.);
\draw [line width=1.pt] (16.,10.)-- (16.,1.);
\draw [line width=1.pt] (4.,10.)-- (8.,10.);
\draw [line width=1.pt] (8.,10.)-- (8.,9.);
\draw [line width=1.pt] (8.,9.)-- (10.,9.);
\draw [line width=1.pt] (10.,9.)-- (10.,7.);
\draw [line width=1.pt] (13.,5.)-- (14.,5.);
\draw [line width=1.pt] (14.,5.)-- (14.,3.);
\draw [line width=1.pt] (14.,3.)-- (16.,3.);
\draw [line width=1.pt] (16.,3.)-- (16.,1.);
\draw [line width=1.pt] (4.,8.)-- (7.,8.);
\draw [line width=1.pt] (11.,6.)-- (11.,4.);
\draw [line width=1.pt] (11.,4.)-- (12.,4.);
\draw [line width=1.pt] (12.,4.)-- (12.,2.);
\draw [line width=1.pt] (12.,2.)-- (14.,2.);
\draw [line width=1.pt] (14.,2.)-- (14.,1.);
\draw [line width=1.pt] (14.,1.)-- (16.,1.);
\draw (3.,4) node[anchor=north west] {$x_i$};
\draw (3.,5) node[anchor=north west] {$x_{i-1}$};
\draw (3.,6) node[anchor=north west] {$x_{i-2}$};
\draw (3,3.) node[anchor=north west,scale=2] {$.$};
\draw (3,3.4) node[anchor=north west,scale=2] {$.$};
\draw (3,3.2) node[anchor=north west,scale=2] {$.$};

\draw (3.,10) node[anchor=north west] {$x_1$};
\draw (3.,9) node[anchor=north west] {$x_2$};
\draw (3,8) node[anchor=north west] {$x_3$};
\draw (3,7.) node[anchor=north west,scale=2] {$.$};
\draw (3,7.4) node[anchor=north west,scale=2] {$.$};
\draw (3,7.2) node[anchor=north west,scale=2] {$.$};

\draw (4.,11) node[anchor=north west] {$y_1$};
\draw (5.,11) node[anchor=north west] {$y_2$};
\draw (6.,11) node[anchor=north west] {$y_3$};
\draw (6.2,11) node[anchor=north west,scale=2] {$...$};

\draw (8.1,11) node[anchor=north west] {$y_j$};
\draw (7.1,11) node[anchor=north west] {$y_{j-1}$};

\draw (9.,11) node[anchor=north west] {$y_{j+1}$};
\draw (10.,11) node[anchor=north west] {$y_{j+2}$};
\draw (11.,11) node[anchor=north west] {$y_{j+3}$};
\draw (15.,11) node[anchor=north west] {$y_m$};
\draw (3.,1.5) node[anchor=north west] {$x_n$};
\draw [line width=1.pt] (10.,7.)-- (12.,7.);
\draw [line width=1.pt] (12.,7.)-- (12.,6.);
\draw [line width=1.pt] (12.,6.)-- (13.,6.);
\draw [line width=1.pt] (13.,6.)-- (13.,5.);
\draw [line width=1.pt] (7.,8.)-- (7.,7.);
\draw [line width=1.pt] (7.,7.)-- (9.,7.);
\draw [line width=1.pt] (9.,7.)-- (9.,6.);
\draw [line width=1.pt] (9.,6.)-- (11.,6.);
\draw [->,line width=.2pt,color=red] (3.8,3.62) -- (12.437647222612357,3.62);
\draw [->,line width=.2pt,color=red] (12.437647222612357,3.6202835631081896) -- (12.423417467073062,5.413232761059219);
\draw [->,line width=.2pt,color=red] (12.423417467073062,5.413232761059219) -- (11.413104823783195,5.399003005519925);
\draw [->,line width=.2pt,color=red] (11.413104823783195,5.399003005519925) -- (11.42733457932249,6.622761981899199);
\draw [->,line width=.2pt,color=red] (11.42733457932249,6.622761981899199) -- (9.378249781664168,6.594302470820612);
\draw [->,line width=.2pt,color=red] (9.378249781664168,6.594302470820612) -- (9.39107282422729,8.571261380902087);
\draw [->,line width=.2pt,color=red] (9.39107282422729,8.571261380902087) -- (8.339462719932767,8.571261380902087);
\draw (4,10) node[anchor=north west] {$1$};
\draw (5,10) node[anchor=north west] {$1$};
\draw (6,10) node[anchor=north west] {$1$};
\draw (6.3,10) node[anchor=north west,scale=2] {$...$};
\draw (7.2,10) node[anchor=north west] {$1$};
\draw (8.2,9) node[anchor=north west] {$1$};

\draw (8.2,10) node[anchor=north west] {$0_R$};
\draw (9.2,10) node[anchor=north west] {$0_R$};
\draw (10.2,10) node[anchor=north west] {$0_R$};
\draw (11.2,10) node[anchor=north west] {$0_R$};
\draw (15,10) node[anchor=north west] {$0_R$};
\draw (15,9) node[anchor=north west] {$0_R$};
\draw (15,8) node[anchor=north west] {$0_R$};
\draw (14,6) node[anchor=north west,scale=2] {$...$};

\draw (13,10) node[anchor=north west,scale=2] {$...$};
\draw (13,11) node[anchor=north west,scale=2] {$...$};
\draw (9.2,9) node[anchor=north west] {$1$};
\draw (7.2,8) node[anchor=north west] {$1$};
\draw (8.2,8) node[anchor=north west] {$1$};
\draw (9.2,8) node[anchor=north west] {$1$};
\draw (6.1,8) node[anchor=north west] {$0_C$};
\draw (5.1,8) node[anchor=north west] {$0_C$};
\draw (4.1,8) node[anchor=north west] {$0_C$};
\draw (10.2,8) node[anchor=north west] {$0_R$};
\draw (11.2,8) node[anchor=north west] {$0_R$};
\draw (10.2,9) node[anchor=north west] {$0_R$};
\draw (11.2,9) node[anchor=north west] {$0_R$};
\draw (13.,8) node[anchor=north west,scale=2] {$...$};
\draw (13.,9) node[anchor=north west,scale=2] {$...$};

\draw (4,9) node[anchor=north west] {$1$};
\draw (5,9) node[anchor=north west] {$1$};
\draw (6,9) node[anchor=north west] {$1$};
\draw (7,9) node[anchor=north west,scale=2] {$...$};
\draw (4.,4) node[anchor=north west] {$0_C$};
\draw (5.,4) node[anchor=north west] {$0_C$};
\draw (6.,4) node[anchor=north west] {$0_C$};
\draw (6.3,4) node[anchor=north west,scale=2] {$...$};
\draw (7.3,4) node[anchor=north west] {$0_C$};
\draw (8.2,4) node[anchor=north west,color=red] {$0_C$};

\draw (8.2,3) node[anchor=north west,scale=2] {$...$};

\draw (9.2,4) node[anchor=north west] {$0_C$};
\draw (10.2,4) node[anchor=north west] {$0_C$};
\draw (11.2,4) node[anchor=north west] {$0_C$};
\draw (12.2,4) node[anchor=north west,scale=1.] {$1$};
\draw (12.5,4) node[anchor=north west,scale=2.] {$...$};
\draw (13.5,4) node[anchor=north west,scale=1.] {$1$};
\draw (11.2,5) node[anchor=north west] {$1$};
\draw (12.2,5) node[anchor=north west] {$1$};
\draw (12.5,5) node[anchor=north west,scale=2.] {$...$};
\draw (13.5,5) node[anchor=north west] {$1$};
\draw (11.2,6) node[anchor=north west] {$1$};
%\draw (11.4,5.9) node[anchor=north west,scale=1.] {$...$};
\draw (12.2,6) node[anchor=north west] {$1$};
\draw (12.35,5.9) node[anchor=north west,scale=1.] {$...$};
\draw (12.6,6) node[anchor=north west] {$1$};

\draw (4.,5) node[anchor=north west] {$0_C$};
\draw (5.,5) node[anchor=north west] {$0_C$};
\draw (6,5) node[anchor=north west] {$0_C$};
\draw (6.3,5) node[anchor=north west,scale=2] {$...$};
\draw (7.3,5) node[anchor=north west] {$0_C$};
\draw (8.2,5) node[anchor=north west] {$0_C$};
\draw (9.2,5) node[anchor=north west] {$0_C$};
\draw (10.2,5) node[anchor=north west] {$0_C$};
\draw (10.2,6) node[anchor=north west] {$0_C$};
\draw (9.2,6) node[anchor=north west] {$0_C$};
\draw (8.2,6) node[anchor=north west] {$0_C$};
\draw (7.3,6) node[anchor=north west] {$0_C$};
\draw (6.,6) node[anchor=north west] {$0_C$};
\draw (5,6) node[anchor=north west] {$0_C$};
\draw (4,6) node[anchor=north west] {$0_C$};
\draw (6.3,6) node[anchor=north west,scale=2] {$...$};
\draw (6.3,7) node[anchor=north west,scale=2] {$...$};

\draw (4.,1.6) node[anchor=north west] {$0_C$};
\draw (5.,1.6) node[anchor=north west] {$0_C$};
\draw (6.,1.6) node[anchor=north west] {$0_C$};
\draw (6.3,1.6) node[anchor=north west,scale=2] {$...$};
\draw (7.2,1.6) node[anchor=north west] {$0_C$};
\draw (8.2,1.6) node[anchor=north west] {$0_C$};
\draw (9.2,1.6) node[anchor=north west] {$0_C$};
\draw (10.2,1.6) node[anchor=north west] {$0_C$};
\draw (11.2,1.6) node[anchor=north west] {$0_C$};
\draw (12,1.6) node[anchor=north west,scale=2] {$...$};
\draw (13.3,1.6) node[anchor=north west] {$0_C$};
\draw (14,1.6) node[anchor=north west] {$1$};
\draw (14.2,1.6) node[anchor=north west,scale=2] {$...$};
\draw (15.3,1.6) node[anchor=north west] {$1$};

\end{tikzpicture}
    \caption{MCA structure of $A(B)$ and a polygonal path from $x_i$ to $y_j$.}
    \label{fig3}
\end{figure}
Similarly, we can show that all the entries to the right of $(x_i,y_j)$ up to the lower stair are $1$. And all the entries above $(x_i,y_j)$ up to the lower stair are $1$. (In the figure 3, we have $P: x_iy_{j+4}x_{i-2}y_{j+3}x_{i-3}y_{j+1}x_2y_j$, $P_1: x_iy_{j+4}x_{i-2}y_{j+3}x_{i-3}y_{j+1}$, $P_2: x_{i-1}y_{j+3}x_{i-3}y_{j+1}x_2y_j$). \\
Next, let $(x_p,y_q)$ be an entry in $A(B)$ which is zero and belongs above the upper stair of MCA. But the entry $(x_p,y_q)$ becomes $1$ in the matrix $A(B^{[k]})$. So there exists a $x_p$-$y_q$ path $P'$ of odd length $\leq k$ in $B$. \\
$P'$: $x_p$ $y_{\beta(p)}$ $x_{\delta\beta(p)}$ $y_{\beta\delta\beta(p)}$ $...$ $x_\frac{\beta\delta\beta\delta...\beta(p)}{ k'_1 times}$ $y_{q}$, where $k'_1+1\leq k$  and  $k'_1+1$ is an odd natural number.\\
Now as before we can show that all the entries to the left and below of $(x_p,y_q)$ up to the upper stair in $A(B^{[k]})$ are ones. Therefore the biadjacency matrix $A(B^{[k]})$ also exibits MCA. Hence $B^{[k]}$, where $k$ is odd natural number, is also a proper interval bigraph.

\end{proof}

\subsection{Powers of chordal bipartite graphs}
A bipartite graph $B$ in which every cycle of length strictly greater than 4 has a chord ( an edge between two non consecutive vertices) is a \textit{chordal bipartite graph}.\\
We have already mentioned that Chandran and Mathew \cite{1} proved that if $G$ is chordal bipartite then so is $G^{[k]}$, for any odd $k\geq 3$. Here we shall prove that chordal bipartite graphs are also strongly closed under powers, which is the main result of this paper.
\begin{theo}
\textit{ Let $G$ be a bipartite graph and if for any odd positive integer $k$, $G^{[k]}$ is chordal bipartite then $G^{[k+2]}$ is also chordal bipartite}.  
\end{theo}
\begin{proof}
Our approach is as follows. We assume that $G^{[k+2]}$ is not chordal bipartite and we shall show that $G^{[k]}$ is also not chordal bipartite. Since $G^{[k+2]}$ is not chordal bipartite, it contains a cycle $C_{2n}$, $ n\geq 3$.
 We consider the following cases.\\
 \textbf{Case 1.} All the edges of $C_{2n}$ belong to $G^{[k+2]}$ but not to $G^{[k]}$. Assume $P_1$, $P_2$, $P_3$,..., $P_{2n}$ are respectively $x_1$-$y_1$, $y_1$-$x_2$, $x_2$-$y_2$,..., $y_n$-$x_1$ shortest path of length $k+2$, where $P_1: x_1y_1'\dots y'_{\alpha}x'_{\alpha}y_1$, $P_2: y_1x''_1\dots x''_{\alpha}y''_{\alpha}x_2$, $P_3: x_2y'''_1\dots y'''_{\alpha}x'''_{\alpha}y_2$,\dots, $P_{2n}: y_nx^{2n}_n\dots x^{2n}_{\alpha}y^{2n}_{\alpha}x_1$. Now in $G^{[k]}$ we have an edge $x_1y'_{\alpha}$ of $G^{[k]}$ followed by two edges $y'_{\alpha}x'_{\alpha}$ and $x'_{\alpha}y_1$ of $G$, similarly an edge $y_1x''_{\alpha}$ of $G^{[k]}$ followed by two edges $x''_{\alpha}y''_{\alpha}$ and $y''_{\alpha}x_2$ and so on. Thus we can construct a cycle $C$ in $G^{[k]}$, where $C: x_1y'_{\alpha}x'_{\alpha}y_1x''_{\alpha}y''_{\alpha}x_2y'''_{\alpha}x'''_{\alpha}y_2\dots x_ny^{(n-1)}_{\alpha}x^{(n-1)}_{\alpha}y_nx^{n}_{\alpha}y^{n}_{\alpha}x_1$ of length $3\times 2n=6n$. Also it is easy to verify that this cycle $C$ does not have any chord of $G^{[k]}$.
\begin{figure}[H]
    \centering
    \begin{tikzpicture}[line cap=round,line join=round,x=1.0cm,y=1.0cm]
\clip(2,0) rectangle (12,10);
\draw [shift={(2.66,6.35)}] plot[domain=-1.06:0.45,variable=\t]({1*1.51*cos(\t r)+0*1.51*sin(\t r)},{0*1.51*cos(\t r)+1*1.51*sin(\t r)});
\draw (4.46,7.56)-- (5,8);
\draw (4,7)-- (4.46,7.56);
\draw [shift={(5.54,9.64)}] plot[domain=4.39:5.65,variable=\t]({1*1.73*cos(\t r)+0*1.73*sin(\t r)},{0*1.73*cos(\t r)+1*1.73*sin(\t r)});
\draw (6.95,8.61)-- (7.51,8.57);
\draw (7.51,8.57)-- (8,8.46);
\draw [shift={(10.32,9.25)}] plot[domain=3.47:4.64,variable=\t]({1*2.45*cos(\t r)+0*2.45*sin(\t r)},{0*2.45*cos(\t r)+1*2.45*sin(\t r)});
\draw (10.14,6.78)-- (10.36,6.3);
\draw (10.36,6.3)-- (10.53,5.74);
\draw [shift={(12.34,4.23)}] plot[domain=2.45:3.55,variable=\t]({1*2.36*cos(\t r)+0*2.36*sin(\t r)},{0*2.36*cos(\t r)+1*2.36*sin(\t r)});
\draw (10.18,3.29)-- (9.8,2.72);
\draw (9.8,2.72)-- (9,2);
\begin{scriptsize}
\fill [color=black] (5,8) circle (1.5pt);
%\draw[color=black] (5.21,8.37) node {$B$};
\fill [color=black] (3.39,5.03) circle (1.5pt);
\draw[color=black] (3.1,5.2) node {$y_n$};
\fill[color=black] (4,7) circle (1.5pt);
\draw[color=black] (4.1,7.5) node {$y_{\alpha}^n$};
\fill [color=black](3.45,5.63) circle (1.5pt);
%\draw[color=black] (3.67,6.01) node {$E$};
\fill[color=black] (3.54,6.02) circle (1.5pt);
%\draw[color=black] (3.72,6.39) node {$F$};
\fill [color=black] (3.72,6.51) circle (1.5pt);
\draw[color=black] (3.6,7) node {$x_{\alpha}^n$};
\fill[color=black] (4.46,7.56) circle (1.5pt);
\draw[color=black] (4.7,8.2) node {$x_1$};
\fill [color=black] (5.62,8.33) circle (1.5pt);
%\draw[color=black] (5.81,8.72) node {$J$};
\fill[color=black] (6.4,8.56) circle (1.5pt);
\draw[color=black] (6.8,8.9) node {$y_{\alpha}'$};
\fill [color=black] (5.98,8.46) circle (1.5pt);
%\draw[color=black] (6.16,8.83) node {$L$};
\fill[color=black] (6.95,8.61) circle (1.5pt);
%\draw[color=black] (7.16,8.99) node {$M$};
\fill [color=black] (7.51,8.57) circle (1.5pt);
\draw[color=black] (7.6,8.8) node {$x_{\alpha}'$};
\fill[color=black] (8,8.46) circle (1.5pt);
\draw[color=black] (8.22,8.7) node {$y_1$};
\fill[color=black] (8.46,8.29) circle (1.5pt);
%\draw[color=black] (8.68,8.67) node {$Q$};
\fill [color=black] (9,8) circle (1.5pt);
%\draw[color=black] (9.22,8.37) node {$R$};
\fill [color=black] (9.48,7.62) circle (1.5pt);
%\draw[color=black] (9.7,7.99) node {$S$};
\fill [color=black] (10.12,6.8) circle (1.5pt);
\draw[color=black] (10.3,7.) node {$x_{\alpha}''$};
\fill [color=black] (10.14,6.78) circle (1.5pt);
%\draw[color=black] (10.33,7.15) node {$V$};
\fill [color=black] (10.36,6.3) circle (1.5pt);
\draw[color=black] (10.6,6.5) node {$y_{\alpha}''$};
\fill [color=black] (10.53,5.74) circle (1.5pt);
\draw[color=black] (10.8,5.8) node {$x_2$};
\fill [color=black] (10.18,3.29) circle (1.5pt);
%\draw[color=black] (10.57,3.66) node {$B_1$};
\fill[color=black] (10.47,4.01) circle (1.5pt);
%\draw[color=black] (10.87,4.39) node {$C_1$};
\fill [color=black] (10.59,4.62) circle (1.5pt);
%\draw[color=black] (11,4.99) node {$D_1$};
\fill [color=black] (10.61,5.04) circle (1.5pt);
%\draw[color=black] (11,5.42) node {$E_1$};
\fill [color=black] (9.8,2.72) circle (1.5pt);
\draw[color=black] (10.5,3.3) node {$y_{\alpha}'''$};
\fill [color=black] (9,2) circle (1.5pt);
\draw[color=black] (10,2.6) node {$x_{\alpha}'''$};
\fill[color=black] (3.45,4.38) circle (1.5pt);
%\draw[color=black] (3.86,4.77) node {$H_1$};
\fill [color=black] (3.58,3.86) circle (1.5pt);
%\draw[color=black] (3.91,4.23) node {$I_1$};
\fill[color=black] (3.75,3.43) circle (1.5pt);
%\draw[color=black] (4.13,3.82) node {$J_1$};
\fill[color=black] (3.97,3.04) circle (1.5pt);
%\draw[color=black] (4.37,3.42) node {$K_1$};
\fill[color=black] (8.53,1.74) circle (1.5pt);
\draw[color=black] (9.3,2.) node {$y_2$};
\fill [color=black] (8.04,1.55) circle (1.5pt);
%\draw[color=black] (8.46,1.93) node {$M_1$};
\fill [color=black] (7.49,1.43) circle (1.5pt);
%\draw[color=black] (7.89,1.82) node {$N_1$};
\end{scriptsize}
\end{tikzpicture}
    \caption{Formation of $C_{6n}$ in $G^{[k]}$ from $C_{2n}$ in $G^{[k+2]}$.}
    \label{fig4}
\end{figure}

\noindent \textbf{Case 2.} Let in $C_{2n}$, $m$ number of edges belong to $G^{[k]}$ but not to $G^{[k-2]}$ and remaining $2n-m$ edges belong to $G^{[k+2]}$ but not to $G^{[k]}$. Now, first we consider the possibility that $m$ edges belong to $G^{[k]}$ but not to $G^{[k-2]}$. Without loss of generality assume $y_1x_2$, $x_2y_2$, $x_4y_4$, \dots
, $y_nx_1$ are the $m$ edges and the remaining $2n-m$ edges of $G^{[k+2]}$ are $x_1y_1$, $y_2x_3$, \dots, $x_ny_n$. Next assume $P_1:x_1y_1'\dots y'_{\alpha}x'_{\alpha}y_1$, $P_2: y_2x''_2\dots x''_{\alpha}y''_{\alpha}x_3$, $\dots$, $P_{2n-m}: x_ny_n^{2n-m}\dots y_{\alpha}^{2n-m}x_{\alpha}^{2n-m}y_n$ are the $2n-m$ paths of length $k+2$. Thus in $G^{[k]}$ we have cycle $C$ of length $3(2n-m)+m=6n-2m=2(3n-m)$. As before we can observe that $C_{2(3n-m)}$ is chordless in $G^{[k]}$.\\
\noindent \textbf{Case 3.} Let in $C_{2n}$, we have $k_1$ edges belong to $G^{[k+2]}$ but not to $G^{[k]}$, $k_2$ edges belong to $G^{[k]}$ but not to $G^{[k-2]}$ and $k_3$ edges belong to $G^{[k-2]}$. Therefore $2n=k_1+k_2+k_3$. Now, $k_1(k+2)+k_2k+k_3(k-2)=k(k_1+k_2+k_3)+2(k_1-k_3)=2nk+2(k_1-k_3)$.\\
First, assume that $2(k_1-k_3)$ is positive even number. Then we may suppose that it is the sum of two odd positive integers of which both are less than $k$ or one is $k$ and other is less than $k$. Thus we have a chordless cycle $C_{2n+2}$ in $G^{[k]}$.\\
Finally, if $2(k_1-k_2)$ is an even negative integer. Then $2nk+2(k_1-k_2)=(2n-1)k+k+2(k_1-k_2)$. Now $k+2(k_1-k_3)$ is odd positive number less than $k$. Thus we have a chordless cycle $C_{2n}$ in $G^{[k]}$. This completes the proof of the theorem.
\end{proof}
\section{Conclusion}
Chandran, Mathew and others \cite{1,Chandran-Ram,Dourisboure} introduced the notion of $k$-chordal graphs. A graph $G$ is $k$-\textit{chordal} if $G$ contains no chordless cycle with more than $k$ vertices. If a bipartite graph $G$ is 4-chordal then it is chordal bipartite. Chandran and Mathew \cite{1} have proved that if a bipartite graph $G$ is $k$-chordal then so is $G^{[m]}$, for any odd natural number $m$. In this context they raise the question: weather if $G^{[m]}$ is $k$-chordal, then $G^{[m+2]}$ is also $k$-chordal ? In this paper we answer this question in affermative for $k=4$ i.e. when $G^{[m]}$ is chordal bipartite graph. Also, in the similar way as the proof of Theorem 5 we can deal with the case where $G^{[m]}$ is $k$- chordal graph  ($k>4$) and $G$ is a bipartite graph.

\end{document}